\documentclass[12pt]{article}
    \usepackage{latexsym}
   \usepackage{amsmath,amsthm,amssymb}
   \usepackage{graphicx}
    \usepackage{here}
    \usepackage{comment}
    \usepackage{color}

\usepackage{url}
      
    \numberwithin{equation}{section}
    \newtheorem{dfn}{Definition}[section]
    \newtheorem{thm}{Theorem}[section]
   \newtheorem{lem}{Lemma}[section]
   \newtheorem{rmk}{Remark}[section]

\begin{document}
\begin{center} 
{\Large \bf New approach for elastic collisions with singular  stress functions}  

\vskip 12pt
Toyohiko Aiki \\
Department of Mathematics, Physics and Computer Sciences, Faculty of Science, \\ Japan Women's University \\
2-8-1 Mejirodai, Bunkyo-ku, Tokyo~112-8681, Japan, \\
{aikit@fc.jwu.ac.jp}

\vskip 12pt
Chiharu Kosugi\\
Graduate School of Sciences and Technology for Innovation,  Yamaguchi University\\
1677-1 Yoshida, Yamaguchi City, Yamaguchi 753-8511,  Japan\\
ckosgi@yamaguchi-u.ac.jp

\end{center}

{\bf Abstract. } 
A collision of a rubber rod to a hard floor is regarded as a simple example of obstacle problems for elastic material.  
In this article  we have proposed a new mathematical model for the collision phenomenon by applying beam equations with singular stress functions, which is investigated in our recent works. As in the works we have established a mathematical method to deal with the singular stress function. Here, we demonstrate the validity of our modeling through observation to the numerical results. Also, we present existence and uniqueness results of the model given as initial boundary value problems.

\section{Modeling}
We consider that we drop a rubber rod onto a hard floor vertically at time $0$. Clearly, the rod collides the floor and bounces off the floor many times. The aim of this paper is to propose a new mathematical model for the collision and to solve it, mathematically.  
As in Figure \ref{fig1}, we suppose that the elastic rod is a one-dimensional material whose natural length is 1. Also, for the original position $x \in [0,1]$ in the natural state, $u = u(t, x)$ denotes the height of $x$ at time $t \in [0,T]$, $0 < T<  \infty$, where the height of the floor surface is 0. Usually, the dynamics of elastic materials are described by deformations, however here we emphasize that we have adopted height, i.e., position, as the main variable to deal with the dynamics. 

\begin{figure}[htbp]
\begin{minipage}{0.31\hsize}
\centering
\includegraphics[width=0.6\textwidth]{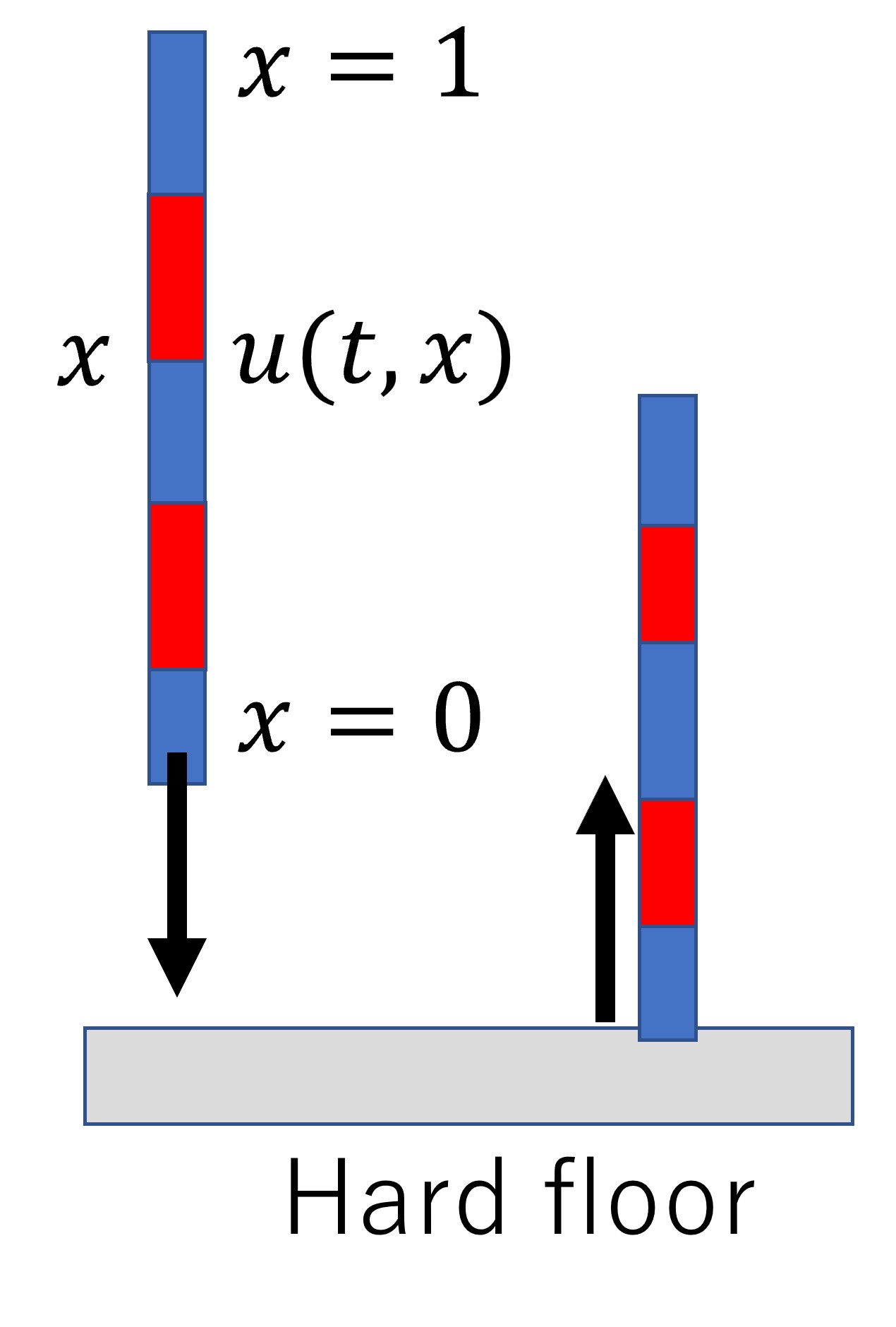}
\caption{Collision of a rubber rod}
\label{fig1}
\end{minipage}
\begin{minipage}{0.31\hsize}
\centering
\includegraphics[width=0.86\textwidth]{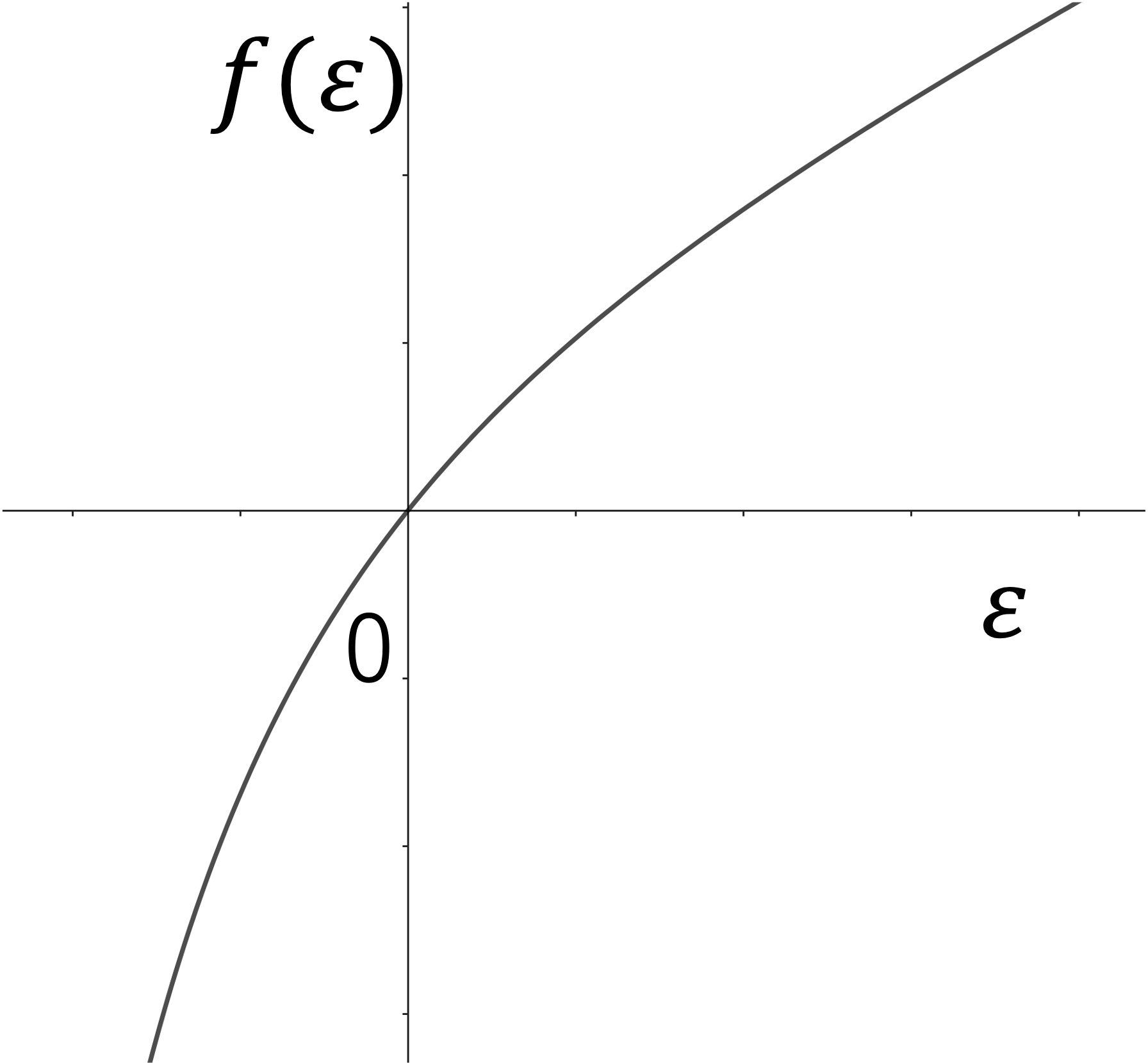}
\caption{Stress function $f$}
\label{fig2}
\end{minipage}
\begin{minipage}{0.31\hsize}
\centering
\includegraphics[width=0.98\textwidth]{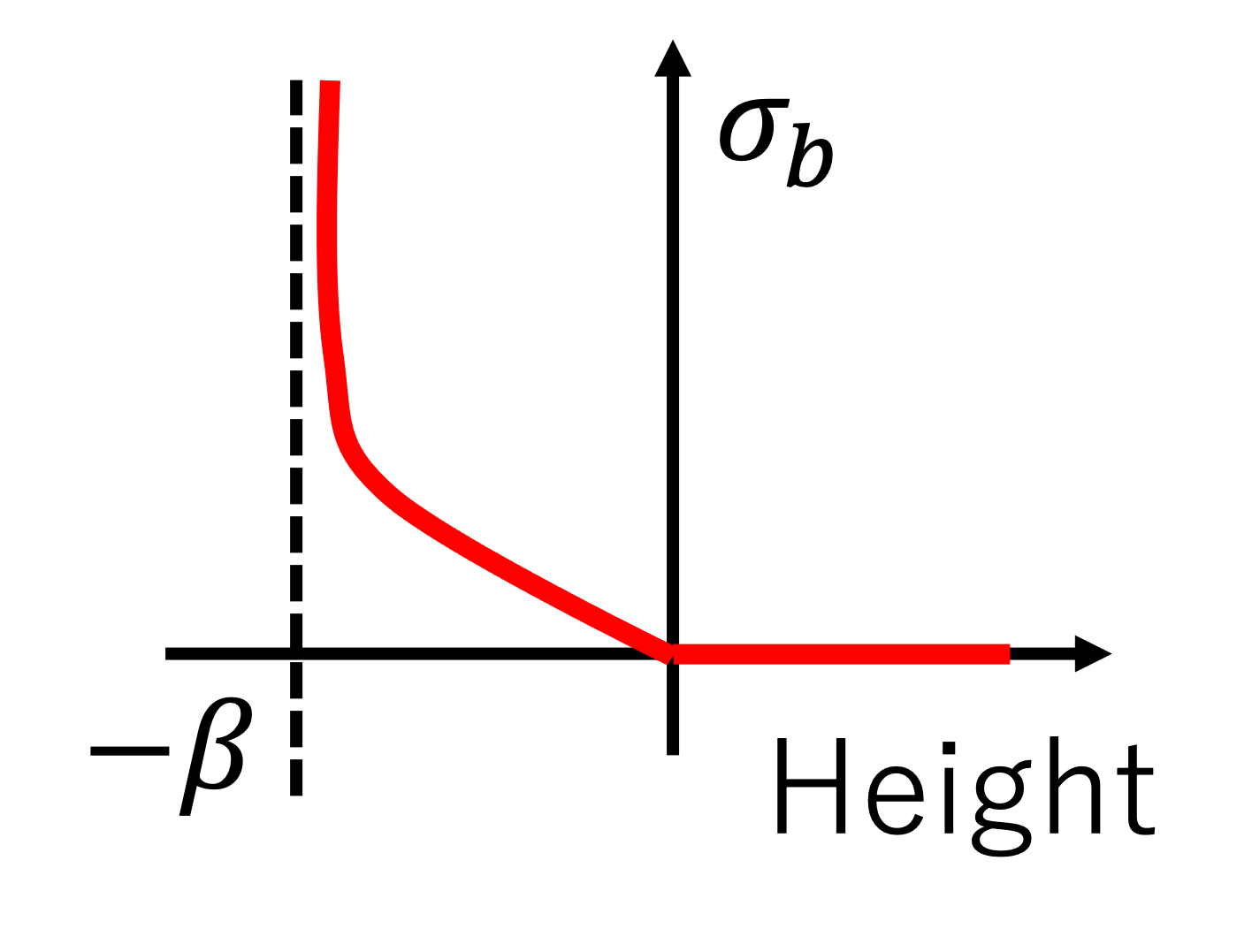}
\caption{$\sigma_b$}
\label{fig3}
\end{minipage}
\end{figure}

Recently, we have investigated the dynamics of elastic materials by applying the stress function having a singularity whose example is  shown in Figure \ref{fig2}, which was proposed in \cite{Aiki-Kro-Muntean} to analyze the mathematical model for the rubber form.  Here, by using the singular functions $f$ and $\sigma_b$ (see Figure \ref{fig3}), we have introduced the initial boundary value problem whose unknown  function is the height $u: Q(T) \to {\mathbb R}$ satisfying
\begin{align}
&u_{tt}  + \gamma u_{xxxx} -   f(\varepsilon)_x 
 - \mu  u_{txx} =  g
   \mbox{ in } Q(T),  \label{EQ} \\
& \gamma u_{xxx} -  f(\varepsilon) -  \mu u_{tx} = 0, u_{xx} = 0
\mbox{ at } x = 1  \mbox{ on }   (0,T),  \label{BC1}\\
& \gamma u_{xxx} -  f(\varepsilon) -  \mu u_{tx}  + \sigma_b(u(t)) = 0, u_{xx}  = 0
\mbox{ at } x = 0 \mbox{ on  }   (0,T),  \label{BC0} \\
& u(0) = u_0,  u_t(0) = v_0 \mbox{ on } (0,1),  \label{IC} 
\end{align}
where $Q(T) = (0,T) \times (0, 1)$,  the strain $\varepsilon$  is given by  $\varepsilon = |u_x|-1$, $\kappa$ is the elastic constant of the rod, $\gamma$ is a positive constant, $\mu$ is the viscosity constant, $g$ is the gravity, $f: (-1,\infty) \to {\mathbb R}$ and $\sigma_b: (-\beta, \infty) \to {\mathbb R}$ are given by 
\begin{align}
 \displaystyle f(\varepsilon) & = \frac{\kappa}{2} (\varepsilon + \frac{1}{2} - \frac{1}{2(1 + \varepsilon)^3}), \label{eqn1-4} \\
\sigma_b(h) & = \left\{ \begin{array}{ll}
            \displaystyle  \kappa_b\left(  \frac{\beta^3}{(h  + \beta)^3  } -1 \right) & \mbox{ for  } h \leq 0, \\
             0 & \mbox{ otherwise,  }  
        \end{array} \right.  \label{sigma_b}
\end{align}
$\beta$ and $\kappa_b$ are positive constants, and $u_0$ and $v_0$ are the initial height and velocity, respectively. 

Our main mathematical result is existence and uniqueness of solutions to the system  P $=$ P$(\mu, u_0,  v_0, f, \sigma_b):=$ \{\eqref{EQ} $\sim$ \eqref{IC}\},  and will be detailed in the next section. 
The derivation of the system P is based on two ideas, application of beam equations with the singular stress function and approximation of the Signorini condition. 
In the rest of this section, we derive the system and present previous results concerning the elastic collisions. 

The beam equation is well-known as a mathematical model of dynamics for elastic materials, and has been extensively studied, for instance Brokate-Sprekels \cite{Bro-Sp},   Racke-Zhang \cite{Racke-Shang} and Takeda-Yoshikawa \cite{Takeda-Yoshikawa2012,Takeda-Yoshikawa2013}.  Also, for shrinking and stretching motion of the elastic curve on the plane ${\mathbb R^2}$ we list our recent results \cite{Aiki-Kosugi,KAAO,Kosugi-Aiki2,Kosugi1,Kosugi2} on the beam equation with the stress function $f \in C((-1, \infty))$ having the  singularity such that $f(\varepsilon) \to  -\infty$ as $\varepsilon \downarrow -1$.  
This type of the stress functions was already applied to mechanics for compressive solid by Ogden \cite{Ogden}, Simo-Miehe \cite{Simo-Miehe} and Holzapfel  \cite{Holzapfel}. In our results  we show that applying the singular stress function enables us to handle large deformations effectively and obtain the lower bound for the strain. We note that the differential equation \eqref{EQ} has the viscosity term $\mu u_{txx}$ for both mathematical and practical purposes. Specifically,  we establish existence and uniqueness results for strong and weak solutions when $\mu > 0$,  as shown in Theorem \ref{main_th} (1) and (2). However, in absence of the viscosity term case ($\mu = 0$), we can prove existence and uniqueness only for a weak solution,  as stated in Theorem \ref{main_th} (3). 

Next, we discuss the boundary conditions \eqref{BC1} and  \eqref{BC0}. At $x=1$, we impose the homogenous Neumann boundary condition, as the rod is free. On the other hand, at $x = 0$, the edge of the rod collides with the hard floor (see Figure \ref{fig1}) and this collision is regarded as one kind of obstacle problems.  The obstacle problem is typically represented by the Signorini condition, 
\begin{equation}
u(t,0 ) \geq 0, u_x(t,0) \geq 0, u(t,0) \cdot u_x(t,0) = 0 \mbox{ for } t \in  [0,T].  \label{Signorini} 
\end{equation} 
The Signorini condition is applied not only to the contact problem for elastic materials but also to a mathematical presentation of unsteady saturated water flow in a porous medium Hornung \cite{Hornung}.  As in  \cite{Hornung}, the Signorini conditions accompanied  by parabolic differential equations were studied by many authors, since we can show the well-posedness of  the problem by applying the monotone operator theory. However, not many results have been obtained for the wave equation. For instance, Lebeau-Schatzman \cite{Lebeau-Schatz} established the existence and uniqueness in the half space, Kim \cite{Kim} proved existence of a weak solution in general domains and several approaches were attempted. Recently, Kashiwabara and Itou \cite{Kashiwabara-Itou} established unique solvability for the linear elasticity with the Signorini condition of dynamic type and Tresca friction condition on the boundary. Roughly, rewriting \eqref{Signorini} to align with their approach yields 
$$
u(t,0 ) + \delta u_t(t,0) \geq 0, u_x(t,0) \geq 0, (u(t,0) + \delta u_t(t,0)) \cdot(u_x(t,0) = 0 \mbox{ for } t \in  [0,T], 
$$
where $\delta$ is  a positive constant. As mentioned as above, the contact problem remains unresolved. 
The primary aim of this article is to propose a novel approach for addressing the contact problem. 
The core concept of this approach is based on the assumption that the normal force arises from the elasticity of the hard floor. 
Under this assumption we impose the boundary condition \eqref{BC0}, where $\sigma_b$ is the stress by bending the floor and $\beta$ denotes the bending limit. Thus, we get the system P. We note that a similar problem with the Lipschitz continuous $\sigma_b$ on $\mathbb R$ was studied in \cite{Aiki-Kro-Muntean} as a mathematical model for rubber foams. 

At the end of this section, we give Figures \ref{fig:a} and \ref{fig:b} representing numerical solutions of  P for $\mu =0$ and $\mu >0$, respectively. In these results we divide the interval to $N$-pieces with $N = 5$ and $\gamma = 0$, and each curve indicates the height of the division points. 
It is evident from the figures that periodic behavior is observed for $\mu =0$, while decay is observed for $\mu > 0$. 
From this observation we infer that $\mu  = 0 $ and $\mu >0$ correspond that the restitution  coefficients  are equal to 1 and less than 1, respectively.  It should be noted here that it is our future problem is to find a more appropriate mathematical description for the viscosity term to represent the energy decay for the collision phenomena. 
 
\begin{figure}[H]
\centering
\begin{minipage}[b]{0.49\columnwidth}
    \centering
    \includegraphics[width=0.9\columnwidth]{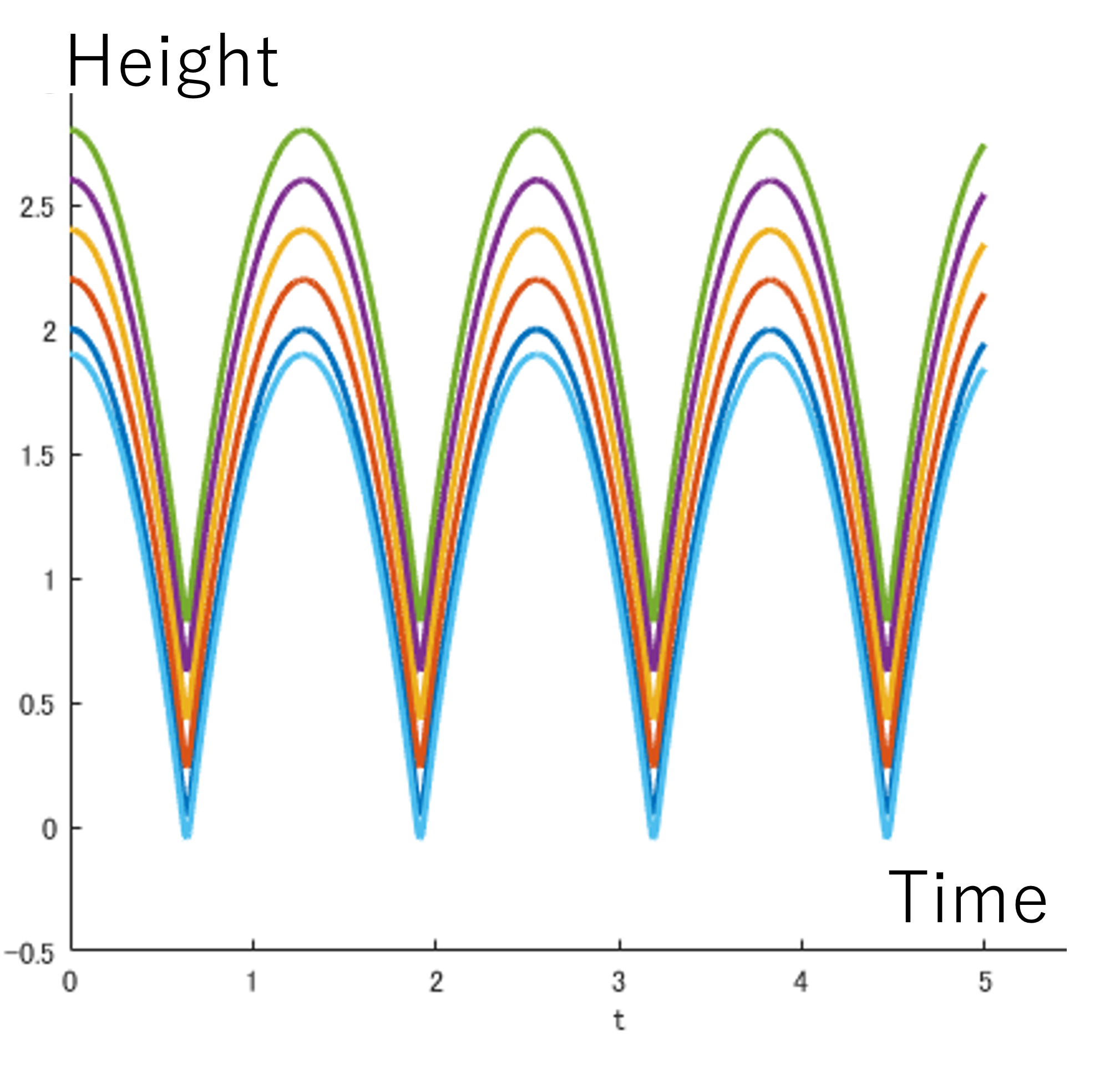}
    \caption{$\mu = 0$.}
    \label{fig:a}
\end{minipage}
\begin{minipage}[b]{0.49\columnwidth}
    \centering
    \includegraphics[width=0.9\columnwidth]{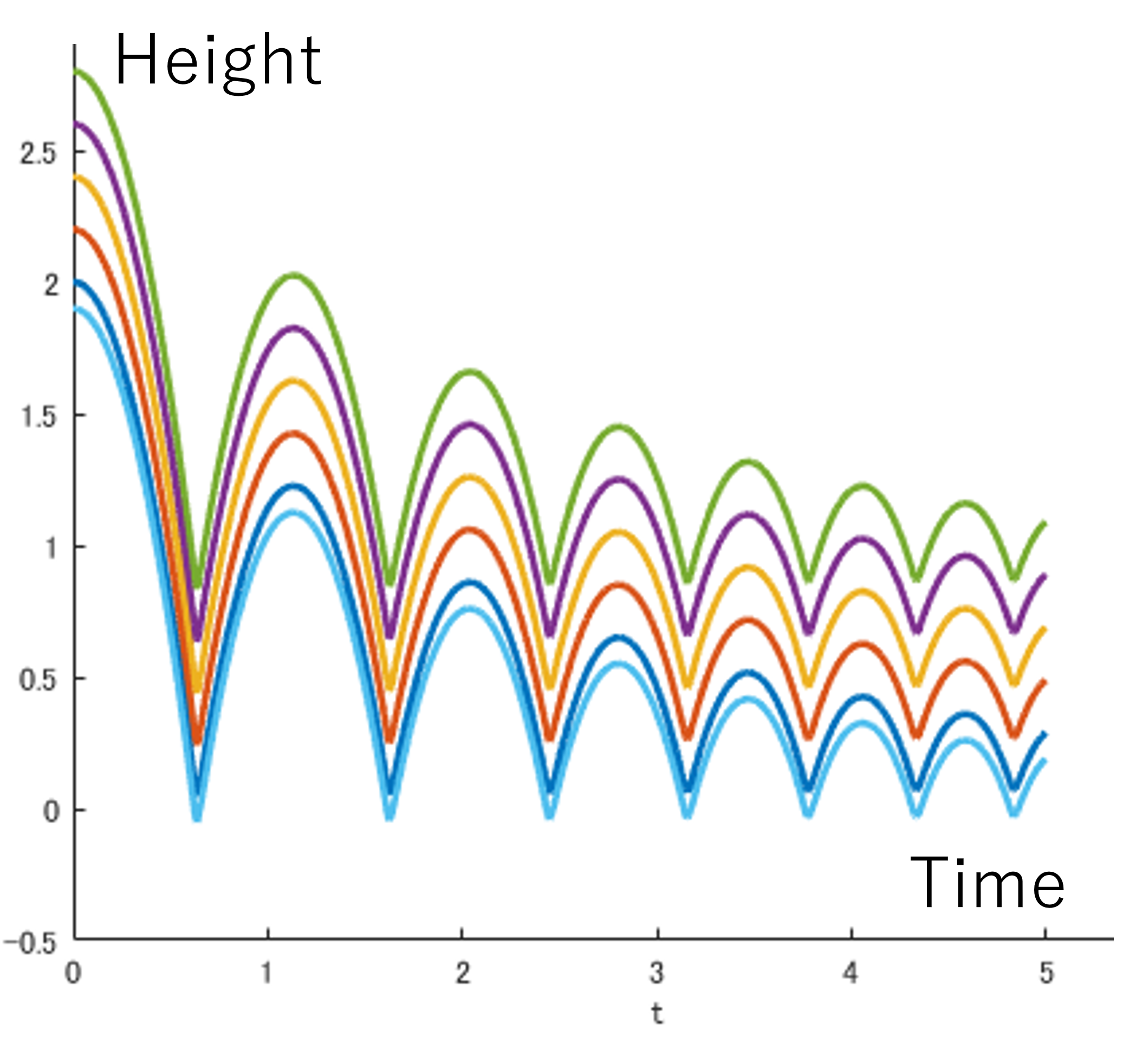}
    \caption{$\mu = 100$.}
    \label{fig:b}
\end{minipage}
\end{figure}

\section{Main result} 
Before we define strong and weak solutions of P$(\mu, u_0,  v_0, f, \sigma_b)$ on $[0,T]$ for $T > 0$ and $\mu \geq 0$, 
we provide an assumption for $f$ and  $\sigma_b$ as follows: 
\\
(A1) $f \in C^1((-1,\infty))$ and  $\sigma_b \in C^1((-\beta, \infty))$ with $\sigma_b = 0$ on $[0,\infty)$ and $\sigma_b \geq 0$ on $(-\beta, 0]$,  where 
$\beta$ is a  positive constant. 
Also, there exists a positive constant $N^*$ such that $f(\varepsilon) \geq 0$ for 
$\varepsilon \geq N^*$. Moreover,  primitives $\hat{f}$ and $\hat{\sigma}_b$ of $f$ and $\sigma_b$ 
with $\hat{f}(0) = \hat{\sigma}_b(0) = 0$, respectively,  satisfy 
\begin{equation}
 \hat{f}(\varepsilon)  \geq -C_f + \frac{c_f}{ (1 + \varepsilon)^2}  \mbox{ for } \varepsilon > -1, 
  \nonumber 
 \end{equation} 
where $C_f$ and $c_f$ are positive constants. Moreover, for any $N >0$ there exists a positive constant $\delta_N$ such that $\hat{\sigma}_b(r) \leq N$ implies $r \geq -\beta + \delta_N$. 

It is clear that for  $f$ and $\sigma_b$ given by \eqref{eqn1-4} and \eqref{sigma_b}, respectively,   (A1) holds.

\begin{dfn} \label{def_strong}  \rm
Let  $T > 0$, $\mu > 0$ and $u$ be a function on $Q(T)$.  We call that the function $u$ is a (strong) solution of P$(\mu, u_0,  v_0, f, \sigma_b)$ on $[0, T]$, if the following conditions (S1)-(S3) hold: 

(S1) $u \in W^{2,\infty}(0,T; L^2(0,1)) \cap W^{2,2}(0,T; H^1(0,1)) \cap   W^{1,\infty}(0,T; H^2(0,1)) \cap$  \\ $L^{\infty}(0, T;  H^4(0,1)) =: S(T)$. 

(S2) $u_x > 0$ on $\overline{Q(T)}$ and $u(t,0) > - \beta$ for $t \in [0,T]$. 

(S3) \eqref{EQ} - \eqref{IC} hold in the usual sense. 

\end{dfn}

\begin{dfn} \label{def_weak}  \rm
Let  $T > 0$,  $\mu > 0$ and $u$ be a function on $Q(T)$.  We call that the function $u$ is a weak solution of P$(\mu, u_0,  v_0, f, \sigma_b)$ on $[0, T]$, if the following conditions (W1)-(W3) hold: 

(W1) $u \in C^1([0,T]; L^2(0,1)) \cap W^{1,2}(0,T; H^1(0,1)) \cap C([0, T] ;  H^2(0,1)) =: W(T)$. 

(W2) $u_x > 0$ on $\overline{Q(T)}$,  $u(t,0) > - \beta$ for $t \in [0,T]$ and $u(0) = u_0$ on $(0,1)$. 

(W3) It holds that
\begin{align}
 & - \int_{Q(T)} u_t \eta_t dxdt + \gamma \int_{Q(T)} u_{xx} \eta_{xx} dxdt + \int_{Q(T)} f(u_x -1) \eta_x dxdt + \mu \int_{Q(T)} u_{tx} \eta_x dxdt  
 \nonumber \\
 = & \int_0^1 v_0 \eta(0) dx  + \int_{Q(T)} g \eta \ dxdt - \int_0^T \sigma_b(u(\cdot, 0)) \eta(\cdot,0) dt  \nonumber 
\end{align} 
for any $\eta \in {\mathcal T}_0(T) := \{ \zeta \in W^{1,2}(0, T; L^2(0,1)) \cap L^2(0,T; H^2(0,1)): \zeta(T) = 0\}$.  
\end{dfn} 

Here, we note that by $W(T) \subset C(\overline{Q(T)})$, (W2) is well-defined for weak solutions. 
Next, we give a characterization of the weak solution.
\begin{rmk} \label{equi_qeak}
Let $u$ be a function on $Q(T)$. $u$ is a weak solution of P$(\mu, u_0,  v_0, f, \sigma_b)$ on $[0, T]$, if and only if 
(W1) and (W2) hold, and $u$ satisfies 
\begin{align}
 & - \int_{Q(T)} u_t \eta_t dxdt  + \int_0^1 u_t(T) \eta (T) dx + \gamma \int_{Q(T)} u_{xx} \eta_{xx} dxdt 
+  \int_{Q(T)} f(u_x - 1) \eta_x dxdt 
 \nonumber \\
& \ + \mu \int_{Q(T)} u_{tx} \eta_x dxdt  
 \nonumber \\
 = & \int_0^1 v_0 \eta(0) dx  + \int_{Q(T)} g \eta \ dxdt - \int_0^T \sigma_b(u(\cdot, 0)) \eta(\cdot,0) dt  \nonumber 
\end{align} 
for any $\eta \in {\mathcal T}(T) :=  W^{1,2}(0, T; L^2(0,1)) \cap L^2(0,T; H^2(0,1))$.  
\end{rmk}

Since $u_t \in C([0,T]; L^2(0,1))$, we can prove this remark, immediately. 

\vskip 12pt 
Moreover, we define a weak solution of  P$(\mu, u_0,  v_0, f, \sigma_b)$ with  $\mu = 0$.  
\begin{dfn} \label{def_weak0}  \rm
 We call that the function $u$ on $Q(T)$ is a weak solution of P$(0, u_0,  v_0, f, \sigma_b)$ on $[0, T]$ for $T > 0$, if the following conditions (W0-1)-(W0-3) hold: 

(W0-1) $u \in W^{1,\infty}(0,T; L^2(0,1))  \cap L^{\infty}(0, T;  H^2(0,1)) =: W_0(T)$. 

(W0-2) For some $\delta >0$,  $u_x \geq \delta $ a.e. on $Q(T)$ and $u(0) = u_0$ on $(0,1)$. 

(W0-3) It holds that
\begin{align}
 & - \int_{Q(T)} u_t \eta_t dxdt + \gamma \int_{Q(T)} u_{xx} \eta_{xx} dxdt + \int_{Q(T)} f(\varepsilon) \eta_x dxdt  \nonumber \\
 = & \int_0^1 v_0 \eta(0) dx + \int_{Q(T)} g \eta \ dxdt  - \int_0^T \sigma_b(u(\cdot, 0)) \eta(\cdot,0) dt \mbox{ for any } \eta \in {\mathcal T}_0(T).   
\end{align}

\end{dfn}


\begin{thm}  \label{main_th} Let $\gamma > 0$, $g \in {\mathbb R}$ and assume (A1). 

(1) Let $\mu > 0$ and assume (AS): 
$$ \leqno{\mbox{  (AS)} }   \left\{ \begin{array}{l}
 u_0 \in H^4(0,1),  v_0 \in  H^2(0, 1),  u_{0x}  > 0 \mbox{  on  } [0,1], u_0(0) > -\beta, 
u_{0xx}(0) = u_{0xx}(1) = 0,  \\
\gamma u_{0xxx}(0)  - \mu v_{0x}(0) + \sigma_b(u_0(0)) = f(u_{0x}(0) -1) , \\
\gamma u_{0xxx}(1)  - \mu v_{0x}(1) = f(u_{0x}(1) - 1). 
\end{array}  \right. 
$$ 
Then the problem  P$(\mu, u_0,  v_0, f , \sigma_b)$ has a strong solution on $[0,T]$. 
 
(2) Let $\mu > 0$. If the following condition (AW) holds, then the problem  P$(\mu, u_0,  v_0, f,   \sigma_b)$ has a unique weak solution on $[0,T]$. 
$$ \leqno{\mbox{  (AW)} }    \quad u_0 \in H^2(0,1), v_0 \in  L^2(0, 1), u_0(0) > -\beta  \mbox{ and } u_{0x}  > 0  \mbox{ on } [0,1]. 
$$

(3) If (AW) holds,  then the problem  P$(0, u_0,  v_0, f, \sigma_b)$ has a unique weak solution on $[0,T]$. 
\end{thm}

\vskip 12pt
From Definitions \ref{def_strong} and \ref{def_weak}, it is obvious that a strong solution of P$(\mu, u_0,  v_0, f,  \sigma_b)$ is a weak solution. Namely, by Theorem \ref{main_th} (2) the uniqueness of the strong solution is valid.   

This paper is organized as follows: 
In the next section  we establish existence and uniqueness of strong and weak solutions to P$(\mu, u_0,  v_0,  f,  \sigma_b)$ for $\mu > 0$. 
Also, the proof of Theorem \ref{main_th} (3) will be given in the last section. 

Throughout this paper,  we put $H = L^2(0,1)$ and its norm is $|\cdot|_H$ for simplicity. Moreover, we list some useful inequalities in our proofs as follows:
\begin{lem} \label{esti_below} (cf. \cite[Lemma 3.2]{Aiki-Kro-Muntean}) 
Let  $z  \in  H^1(0,1)$, and $r_1,  r_2$ be positive constants.  If 
$$ \int_0^1 \frac{1}{|z|^2} dx \leq r_1, \quad |z|_{H^1(0,1)} \leq r_2, $$
then it holds 
$$ |z| \geq \frac{r_2}{\sqrt{2}} \exp(- r_1 r_2^2) \quad \mbox{ on } [0,1]. $$
\end{lem}
By this lemma, we can deal with the singular stress function $f$. The next lemma is concerned with the embedding for Sobolev spaces. 
\begin{lem} \label{GN_ineq} 
For some positive constant $C_0$ it holds that
\begin{align*}
 & |z|  \leq C_0( |z_x|_H^{1/2} |z|_H^{1/2} + |z|_H) \mbox{ on } [0,1]  \mbox{ for any } z \in H^1(0,1), \\
 & |z_x|_H   \leq C_0(|z_{xx}|_H + |z|_H)  \mbox{ for any } z \in H^2(0,1). 
\end{align*}
\end{lem}

\section{Problem with viscosity} \label{viscosity}
In order to show existence of solutions to P$(\mu, u_0,  v_0,  f,  \sigma_b)$ for $\mu > 0$ we introduce the following linear problem (LP)($\mu, u_0, v_0, g, F, q$) for $\mu > 0$: 
\begin{align}
&u_{tt}  + \gamma u_{xxxx}   - \mu  u_{txx} =   g + F_x
   \mbox{ in } Q(T),  \label{EQA1} \\
& \gamma u_{xxx} -  \mu u_{tx} = F, u_{xx} = 0
\mbox{ at } x = 1  \mbox{ on }  (0,T),  \label{BCA1}\\
& \gamma u_{xxx}  -  \mu u_{tx}  +  q = F, u_{xx} = 0
\mbox{ at } x = 0 \mbox{ on }  (0,T),  \label{BCA0} \\
& u(0) = u_0,  u_t(0) = v_0 \mbox{ on } (0,1),  \label{ICA} 
\end{align}
where $g$ is a number, 
$u_0$ and  $v_0$ are initial functions on $[0,1]$,  and $F$ and $q$ are given functions defined on $Q(T)$ and $[0, T]$, respectively.  
 
The aim of this section is to prove  existence and uniqueness of solutions to P$(\mu, u_0,  v_0,  f,  \sigma_b)$ for $\mu > 0$. 
As a first step in the proof, we provide the following lemma which guarantees existence and uniqueness of weak and strong solutions of (LP)($\mu, u_0, v_0, g, F, q$) on $[0,T]$. 

\begin{lem} \label{lem_aux}
Let $\gamma > 0$ and $g \in {\mathbb R}$.  

(1) Let  $\mu > 0$. If $u_0 \in H^2(0,1)$, $v_0 \in H$, $F \in L^2(0, T; H)$ and $q \in L^2(0, T)$, then the problem  (LP)$(\mu, u_0,  v_0, g, F, q)$ has a unique weak solution $u \in W(T)$ on $[0,T]$.  Namely, $u \in W(T)$, $u(0) = u_0$ and 
\begin{align}
 & - \int_{Q(T)} u_t \eta_t dxdt + \gamma \int_{Q(T)} u_{xx} \eta_{xx} dxdt 
 + \mu \int_{Q(T)} u_{tx} \eta_x dx dt  - \int_{Q(T)} F \eta_x dxdt  \nonumber \\
 = & \int_0^1 v_0 \eta(0) dx + \int_{Q(T)} g \eta \ dxdt  - \int_0^T  q \eta(\cdot,0) dt \quad \mbox{ for } \eta \in {\mathcal T}_0(T).  
  \nonumber 
\end{align} 

Moreover, it holds that 
\begin{align}
& \frac{1}{2} \frac{d}{dt} |u_t|_H^2 + \frac{\gamma}{2} \frac{d}{dt} |u_{xx}|_H^2 
 + \mu |u_{tx}|_H^2 \nonumber \\
 = & \int_0^1 g u_t dx   - \int_0^1 F u_{tx} dx + q u_t(\cdot, 0) \quad \mbox{ a.e. on } [0,T].  \label{lem3-1-1} 
 \end{align}
 
 (2) Let $\mu > 0$, $u_0 \in H^4(0,1)$,  $v_0 \in  H^2(0, 1)$,   $u_{0xx}(0) = u_{0xx}(1) = 0$, $F \in W^{1,2}(0, T; H) \cap L^{\infty}(0, T; H^1(0,1))$ and $q \in W^{1,2}(0,T)$. If $\gamma u_{0xxx}(0)  - \mu v_{0x}(0) + q(0) = F(0,0)$ and 
$\gamma u_{0xxx}(1)  - \mu v_{0x}(1) = F(0,1)$, then the problem  (LP)$(\mu, u_0,  v_0, g, F, q)$ has a unique strong solution $u \in S(T)$ on $[0,T]$. 
\end{lem}

\vskip 12pt
We note that the problem  (LP)($\mu, u_0, v_0, g,  F, q$) is linear,  and \eqref{lem3-1-1} and \eqref{lem3-1-1} are easily obtained. Consequently, Lemma \ref{lem_aux} can be proved by the standard way, for instance, the time-discretizatoin method (see \cite{Bro-Sp} and \cite{Kosugi(Thesis)}). So, we omit its proof.

\vskip 12pt 
In order to prove Theorem \ref{main_th} (1) and (2) by applying the Banach fixed point theorem, 
(AP)($\mu, u_0, v_0, \tilde{u}, f_N, \sigma_{bN}$) denotes (LP)$(\mu, u_0,  v_0, g, f_N(\tilde{u}_x-1), \sigma_{bN}(\tilde{u}(\cdot,0)))$ for
$N > 0$ and a given function $\tilde{u}$ on $Q(T)$, where $f_N$ and $\sigma_{bN}$ are trancations of $f$ and $\sigma_b$, respectively, that is,
$$
f_N(r) = \left\{ \begin{array}{ll} f(N) & \mbox{ for } r \geq N, \\ f(r) & \mbox{ for } -1 < r < N, \end{array} \right. 
\sigma_{bN}(r) = \left\{ \begin{array}{ll} \sigma_b(N) & \mbox{ for } r \geq N, \\ \sigma_b(r) & \mbox{ for } -\beta < r < N,  \end{array} \right. 
\mbox{ for } r \in{\mathbb R}.
$$ 
Here, we introduce a class of functions $\tilde{u}$ such that the problem (AP)($\mu, u_0, v_0, \tilde{u},  f_N, \sigma_{bN}$) has a weak solution as follows:  For $T > 0$ and  $\delta_0 > 0$ we put 
$$ K(\delta_0, T) = \{ \tilde{v} \in X(T) : \tilde{v}_x \geq \delta_0 \mbox{ a.e. on } Q(T) \mbox{ and } \tilde{v}(t, 0 ) \geq - \beta + \delta_0 \mbox{ for a.e.  } t \in [0,T] \}, $$
where $X(T) =  L^2(0, T; H^2(0,1)) \cap W^{1,2}(0,T; H)$. 
Clearly, Lemma \ref{lem_aux} (1) implies the existence and uniqueness of a weak solution $u$ to  (AP)($\mu, u_0, v_0, \tilde{u},  f_N, \sigma_{bN}$)  on $[0, T]$ for $\tilde{u} \in K(\delta_0, T)$ and $\delta_0 > 0$,  since $f_N$ and $\sigma_{bN}$ are Lipschitz continuous on $[\delta_0 - 1, N]$ and $[\delta_0 - \beta, N]$, respectively. 
From now on, we give some lemmas concerned with uniform estimates for solutions of (AP)($\mu, u_0, v_0, \tilde{u},  f_N, \sigma_{bN}$). 
For simplicity, let $\Gamma_N$ be the solution operator of (AP)($\mu, u_0, v_0, \tilde{u},  f_N, \sigma_{bN}$), namely, $\Gamma_N(\tilde{u}) = u$ for $\tilde{u} \in K(\delta_0, T)$. 

\begin{lem}  \label{esti_1}
Let $\gamma > 0$, $\mu > 0$, $g \in {\mathbb R}$, $\delta_0 > 0$ and $N > 0$, and assume (A1). 
Then, $\Gamma_N(\tilde{u}) \in X(T)$ for any $\tilde{u} \in K(\delta_0, T)$. 
Also, there exists a positive constant $C_1$ such that  
 \begin{align}
   & |u_t(t)|_H^2 + |u_{xx}(t)|_H^2 + \int_0^t |u_{x\tau}|_H^2 d\tau  \nonumber \\
\leq &  C_1 (|v_0|_H^2 + |u_{0xx}|_H^2) + C_1 \int_0^t (|f_N(\tilde{u}_x(\tau) - 1)|_H^2 + |\sigma_{bN}(\tilde{u}(\tau,0))|^2  +1) d\tau
 \label{Lem3-3-1}
\end{align}
for $0 \leq t \leq T$ and $\tilde{u} \in K(\delta_0, T)$, where $u = \Gamma_N(\tilde{u})$.

\end{lem} 

\begin{proof} On account of  Lemma \ref{lem_aux} (1), we have $\Gamma_N(\tilde{u}) \in X(T)$ for any $\tilde{u} \in K(\delta_0, T)$. 
In addition, by applying Gronwall's inequality to \eqref{lem3-1-1}, this lemma is easily proved. 
\end{proof} 

Here, due to Banach's fixed point theorem to $\Gamma_N$ we shall show local existence of a weak solution of P$_{\mu}^N(u_0, v_0) :=$ P$(\mu, u_0,  v_0, f_N, \sigma_{bN})$ for $\mu > 0$ and $N > 0$ as the following lemma. 
\begin{lem} \label{local_existence} 
Let $\gamma > 0$, $\mu > 0$, $g \in {\mathbb R}$ and $N  > 0$, and assume (A1). If the condition (AW) holds, then there exists a unique weak solution $u$ of P$_{\mu}^N(u_0, v_0)$ on $[0,T']$ for some $0 < T' \leq T$.  Also, $T'$ is depending only on $|u_0|_{H^2(0,1)}$, $|v_0|_{H}$, $\min_{x \in [0,1]} u_{0x}(x)$ and $u_0(0)$. 
\end{lem}

\begin{proof} 
First, thanks to (AW) we can take $\delta > 0$ such that $u_{0x} \geq 2 \delta$ on $[0,1]$ and $u_0(0) \geq -\beta + 2\delta$, and put
$$ K_1(\delta, T_1, M) = \{ \tilde{u} \in K(\delta, T):  |\tilde{u}_{xx}|_H \leq M \mbox{ and } |\tilde{u}_t|_H \leq M \mbox{ a.e. on }  [0,T_1]\} $$ 
for $M > 0$ and $0 < T_1 \leq T$, which are chosen as suitable positive constants, later. 
It is clear that $K_1(\delta, T_1, M)$ is closed in $X(T_1)$ for $\mu > 0$ and $N > 0$.

Since there exists a positive constant $R_1(M)$ depending on $M$ such that 
$$ |\tilde{u}|,   |\tilde{u}_x| \leq R_1(M) \quad \mbox{ a.e. on } Q(T_1) \mbox{ for } \tilde{u} \in K_1(\delta, T_1, M), $$
\eqref{Lem3-3-1} implies existence of some positive constant $R_2(M)$ satisfying 
\begin{align}
& |\Gamma_N(\tilde{u})_t(t)|_H^2 + |\Gamma_N(\tilde{u})_{xx}(t)|_H^2 + \int_0^t |\Gamma_N(\tilde{u})_{x\tau}|_H^2 d\tau   \nonumber \\
\leq & C_1 (|v_0|_H^2 + |u_{0xx}|_H^2) + R_2(M) t \quad \mbox{ for } 0 \leq t \leq T_1 \mbox{ and } \tilde{u} \in K_1(\delta, T_1,  M).  \nonumber 
\end{align} 
Accordingly, there exist positive constants $M > 0$ and $T_1 \in (0,T]$ such that 
\begin{equation}
 |\Gamma_N(\tilde{u})_t(t)|_H^2 + |\Gamma_N(\tilde{u})_{xx}(t)|_H^2 + \int_0^t |\Gamma_N(\tilde{u})_{x\tau}|_H^2 d\tau  
\leq M^2  \mbox{ for } 0 \leq t \leq T_1 \mbox{ and } \tilde{u} \in K_1(\delta, T_1,  M). \label{3_002} 
\end{equation}
Here, thanks to \ref{GN_ineq}, it is easy to see that 
\begin{align*}
\Gamma_N(\tilde{u})(t, 0) & \geq  u_{0}(0)  - \int_0^t  |\Gamma_N(\tilde{u})_{\tau}(\tau,0)| d\tau \\
 &\geq  -\beta + 2 \delta - 2C_0 \int_0^t  ( |\Gamma_N(\tilde{u})_{\tau x}|_H + |\Gamma_N(\tilde{u})_{\tau}(\tau)|_H) d\tau \\
 & \geq  -\beta + 2 \delta - 2C_0  (t M + t^{1/2} M) \quad  \mbox{ for } 0 \leq t \leq T \mbox{ and } \tilde{u} \in K(\delta, T,  M). 
\end{align*}
Also, we have 
\begin{align*}
 & \Gamma_N(\tilde{u})_x(t, x)  \\
 \geq & \  u_{0x}(x) -   |\Gamma_N(\tilde{u})_x (t, x)  - u_{0x}(x)| \\
   \geq &  \ 2 \delta - C_0 ( |\Gamma_N(\tilde{u})_x(t)  - u_{0x}|_H^{1/2}  |\Gamma_N(\tilde{u})_{xx}(t)  - u_{0xx}|_H^{1/2} 
   +  |\Gamma_N(\tilde{u})_x(t)  - u_{0x}|_H)    \\ 
 \geq &  \ 2 \delta - C_0 ( | \int_0^t (|\Gamma_N(\tilde{u})_{x\tau}|_H d\tau|^{1/2}  |\Gamma_N(\tilde{u})_{xx}(t)  - u_{0xx}|_H^{1/2}
  +  | \int_0^t (|\Gamma_N(\tilde{u})_{x\tau}|_H d\tau|)  \\
 \geq &  \ 2 \delta - C_0 ( t^{1/4} M^{1/2} (M^{1/2} + |u_{0xx}|_H^{1/2}) + t^{1/2} M^{1/4} )
  \mbox{ for } (t, x)  \in Q(T_1) \mbox{ and } \tilde{u} \in K_1(\delta, T_1,  M). 
\end{align*}
Consequently, by \eqref{3_002} we can take 
$T_2 \in (0, T_1]$ such $\Gamma_N(\tilde{u}) \in K_1(\delta, T_2, M)$ for any $\tilde{u} \in K_1(\delta, T_2,  M)$, namely, 
$\Gamma_N:  K_1(\delta, T_2, M) \to K_1(\delta, T_2, M)$.  

To complete the proof of existence of a weak solution, we showthat  $\Gamma_N$ is contractive on $K_1(\delta, T_3, M)$ with respect to the norm of  $X(T_3)$ for some $T_3 \in (0, T_2]$. Let $\tilde{u}_i \in K_1(\delta, T_2, M)$ and $u_i = \Gamma_N(\tilde{u}_i)$ for $i = 1, 2$. 
For simplicity,  put $\tilde{u} = \tilde{u}_1 - \tilde{u}_2$ and $u = u_1 - u_2$. Then, $u$ is a weak solution of 
(LP)$(\mu, 0, 0, 0, F_{\ast}, q_{\ast})$ on $[0,T]$, where $F_{\ast} = f_N(\tilde{u}_{1x} -1) - f_N(\tilde{u}_{2x} -1)$ and 
$q_{\ast} = \sigma_{bN}(\tilde{u}_1(\cdot,0) )-  \sigma_{bN}(\tilde{u}_2(\cdot,0) )$ on $[0,T]$. 
Accordingly, by  \eqref{lem3-1-1} we have 
\begin{align*}
 \frac{1}{2} \frac{d}{dt} |u_t|_H^2 + \frac{\gamma}{2} \frac{d}{dt} |u_{xx}|_H^2 
 + \mu |u_{tx}|_H^2
 =    - \int_0^1 F_{\ast} u_{tx} dx + q_{\ast} u_t(\cdot, 0) \quad \mbox{ a.e. on } [0,T].  
 \end{align*}
Because of (A1) and the definition of $K_1(\delta, T_2, M)$, there exits a positive constant $C_3$ satisfying
$|F_{\ast}| \leq C_3 |\tilde{u}_x| \mbox{ a.e. on } Q(T)$,  $|q_{\ast}| \leq C_3 |\tilde{u}(\cdot, 0)| \mbox{ a.e. on } (0,T) \mbox{ for }
\tilde{u}_1, \tilde{u}_2 \in K_1(\delta, T_2, M)$. 
From the argument above with help of Lemma \ref{GN_ineq}, we infer that 
\begin{align}
   \frac{1}{2} \frac{d}{dt} |u_t|_H^2 + \frac{\gamma}{2} \frac{d}{dt} |u_{xx}|_H^2 
 + \frac{\mu}{2} |u_{tx}|_H^2  
 \leq & \     C_4 (|\tilde{u}_{x}|_H^2 + |\tilde{u}(\cdot, 0)|^2 + |u_t|_H^2)  \nonumber \\
\leq  & \ C_5 (|\tilde{u}_{xx}|_H^2 + |\tilde{u}|_H^2 + |u_t|_H^2)   \quad \mbox{ a.e. on } [0,T],   \label{eqn_3-9}
 \end{align}
where $C_4$ and $C_5$ are suitable positive constants.  By applying Gronwall's inequality to \eqref{eqn_3-9}, it is easy to see that  for some positive constants $C_6$ and $C_7$ the following inequalities  hold:  
\begin{align*}
&  |u_t(t)|_H^2 + |u_{xx}(t)|_H^2  + \int_0^t  |u_{\tau x}|_H^2 d\tau \leq C_6 M^2 t \quad  \mbox{ for } 0 \leq t \leq T_2, \\ 
&  |u|_{X(T_3)} \leq C_7 T_3 |\tilde{u}|_{X(T_3)} \quad \mbox{ for } 0 \leq T_3 \leq T_2, 
\end{align*}
where $|u|_{X(T_3)} = |u|_{L^2(0,T_3; H^2(0,1))} + |u_t|_{L^2(0, T_3; H)}$. Choosing a small $T_3 \in (0,T_2]$, we can apply Banach's fixed point theorem to $\Gamma_N$. Hence existence and uniqueness of weak solutions to P$_{\mu}^N$ hold on $[0,T_3]$ for $N > 0$ and $\mu > 0$. Thus, we have proved this lemma. 
\end{proof} 

From now on, we give a proof for existence and uniqueness of weak solutions of P$_{\mu}$ for $\mu > 0$.  
\begin{proof}[Proof of Theorem \ref{main_th} (2)] Let $\mu > 0$ and $N \geq N^*$. First, by (A1) is clear that 
\begin{equation} 
\hat{f}_N(\varepsilon) \geq - C_f + \frac{c_f}{(1+ \varepsilon)^2}   \mbox{ for } \varepsilon > -1, 
\hat{\sigma}_{bN}(r) \leq 0 \mbox{ for } r  > -\beta,   \label{eqn_c10}
\end{equation} 
where $\hat{f}_N$ and $\hat{\sigma}_{bN}$ are primitives of $f_N$ and $\sigma_{bN}$, respectively, with $\hat{f}_N(0) = \hat{\sigma}_{bN}(0) = 0$.

Due to Lemma \ref{local_existence}, for some $T' \in (0,T]$  P$_{\mu}^N(u_0, v_0)$ admits a weak solution $u$ on $[0,T']$. 
Obviously,  by \eqref{lem3-1-1} in  Lemma \ref{lem_aux} (1),  we have 
\begin{align}
& \frac{1}{2} \frac{d}{dt} |u_t(t)|_H^2 + \frac{\gamma}{2} \frac{d}{dt} |u_{xx}(t)|_H^2 
 + \mu |u_{tx}(t)|_H^2 \nonumber \\
 = & \int_0^1 g u_t (t) dx   - \int_0^1 f_N(u_x(t) - 1) u_{tx}(t) dx +\sigma_{bN}(u(t,0)) u_t(t, 0)  \nonumber \\
= &  \frac{d}{dt} (\int_0^1 g u(t) dx - \int_0^1 \hat{f}_N(u_x(t) -1) dx + \hat{\sigma}_{bN}(u(t,0)))  
 \quad \mbox{ for  a.e. } t \in [0,T'].  \label{3-4-00} 
 \end{align}
Integrate \eqref{3-4-00}, and then, we see that 
\begin{align*}
&  \frac{1}{2} |u_t(t)|_H^2 + \frac{\gamma}{2} |u_{xx}(t)|_H^2 + \mu \int_0^t |u_{\tau x}|_H^2 d\tau + \int_0^1 \hat{f}_N(u_x(t) -1) dx 
 - \hat{\sigma}_{bN}(u(t,0)) \\
=  &  \frac{1}{2} |v_0|_H^2 + \frac{\gamma}{2} |u_{0xx}|_H^2 + \int_0^1 \hat{f}_N (u_{0x} -1) dx 
-  \hat{\sigma}_{bN}(u_0(0)) + g\int_0^1 ( u(t) - u_0) dx  \mbox{ for  } 0 \leq t \leq T'. 
\end{align*}
Also, it is clear that 
$$ \int_0^1 u(t) dx  \leq |u_0|_H + \int_0^t |u_{\tau}|_H d\tau \mbox{ for } 0 \leq t \leq T'. $$ 
On account of (AW), we can take $N_1$ such that $|u_{0x}| \leq N_1$ on $[0,1]$ and $u_0(0) \leq N_1$. As a result, 
$f_N(u_{0x}-1) = f(u_{0x} - 1)$ and $\hat{\sigma}_{bN}(u_0(0)) = \hat{\sigma}_{b}(u_0(0))$ 
for $N \geq N_1$. Here, by putting 
$$ C(u_0, v_0)  = \frac{1}{2} |v_0|_H^2 + \frac{\gamma}{2} |u_{0xx}|_H^2 + \int_0^1 \hat{f}(u_{0x} -1) dx 
-  \hat{\sigma}_{b}(u_0(0)) +  g (|u_0|_H - \int_0^1  u_0 dx),  $$ 
 \eqref{eqn_c10} guarantees that 
\begin{align}
&  \frac{1}{2} |u_t(t)|_H^2 + \frac{\gamma}{2} |u_{xx}(t)|_H^2 + \mu \int_0^t |u_{\tau x}|_H^2 d\tau + c_f \int_0^1 \frac{1}{|u_x(t)|^2} dx 
- \hat{\sigma}_{bN}(u(t,0)) \nonumber \\
\leq  & C(u_0, v_0) + C_f + g \int_0^t  |u_{\tau}|_H  d\tau  \mbox{ for  } 0 \leq t \leq T' \mbox{ and } N \geq N_2 := \max\{N^*, N_1\}.  \label{3-9} 
\end{align}
Accordingly, thanks to Gronwall's inequality, there exists a positive constant $C_7$ independent of $N \geq N_2$ such that 
\begin{equation}
 |u_t(t)|_H^2 +  |u_{xx}(t)|_H^2 + \mu \int_0^t |u_{\tau x}|_H^2 d\tau +  \int_0^1 \frac{1}{|u_x(t)|^2} dx 
- \hat{\sigma}_{bN}(u(t,0)) \leq C_7 \quad \mbox{ for } 0 \leq t \leq T'. \label{3-4-001}
\end{equation}
We note that $C_7$ depends on $T$, but is independent of $T'$. Also,  Lemma \ref{esti_below} implies existence of a positive constant $\delta_1$ satisfying $u_x \geq \delta_1$  on $\overline{Q(T')}$. Moreover, from (A1) and $- \hat{\sigma}_{bN}(u(t,0)) \leq C_7$ for $t \in [0,T']$ it follows existence of a positive constant $\delta_2$ such that $u(t,0) \geq  -\beta + \delta_2$ for $t \in [0,T']$.  In fact, we may suppose that $u(t,0) < 0$. Clearly, thanks to (A1), 
$\hat{\sigma}_{bN}(u(t,0)) = \hat{\sigma}_{b}(u(t,0))$ and we get the required $\delta_2$.   From these estimate, we observe that the pair $\{u(T'), u_t(T')\}$ satisfies (AW). Hence, by Lemma \ref{local_existence} we can get a weak solution of P$_{\mu}^N(u(T'),  v(T'))$ on $[T', T'+ \hat{T}]$ for some $\hat{T} \in (0, T]$. On account of Remark \ref{equi_qeak}, we have a weak solution $u$ of P$_{\mu}^N(u_0,  v_0)$ on $[0,T' + \hat{T}]$. 
Since the estimate \eqref{3-4-001} holds for $t = T' + \hat{T}$, we can extend the weak solution on to $[0, T+ 2\hat{T}]$ as mentioned in Lemma \ref{local_existence}.  Therefore, by repeating this argument finite times, existence of a weak solution of P$_{\mu}^N(u_0,  v_0)$ on $[0,T]$. Moreover, from \eqref{3-4-001} we can take a positive constant $N_3$ such that $u_x \leq N_3$ on $\overline{Q(T)}$ and $u(\cdot, 0) \leq N_3$ on $[0,T']$. Therefore, for $N \geq N_3$, we have 
$f_N(u_x -1) = f(u_x -1)$ on $\overline{Q(T)}$ and $\sigma_{bN}(u(\cdot, 0)) = \sigma_b(u(\cdot, 0))$ on $[0, T]$. Namely, existence of weak solution of P$(\mu, u_0,  v_0)$ on $[0,T]$ is proved. 

The uniqueness of weak solutions can be obtained from \eqref{lem3-1-1} in Lemma \ref{lem_aux}. Indeed, Let $u_1$ and $u_2$ be weak solutions of  P$(\mu, u_0,  v_0)$ on $[0,T]$ and put $u = u_1 - u_2$. It is clear that $u$ is a weak solution of LP$(\mu, 0, 0, 0, F, q)$ on $[0,T]$, where 
$F = f(u_{1x} - 1) - f(u_{2x}-1)$  in $\overline{Q(T)}$ and $q = \sigma_b(u_1(\cdot, 0))- \sigma_b(u_2(\cdot, 0))$ on $[0,T]$. By Definition \ref{def_weak}, we can choose $\delta > 0$ and $M > 0$ such that 
$$ -1 + \delta \leq u_{ix} -1 \leq M \mbox{ on } \overline{Q(T)} \mbox{ and } -\beta + \delta \leq u_i(\cdot,0) \leq M \mbox{ on } [0, T] \mbox{ for } i = 1, 2. $$
Hence,  since we can regard that $f$ and $\sigma_b$ are Lipschitz continuous, Theorem \ref{main_th} (2)  is a direct consequence of  \eqref{lem3-1-1} in Lemma \ref{lem_aux} and Gronwall's inequality. 
\end{proof}

\vskip 12pt
Next, we prove existence of a strong solution. 
\begin{proof}[Proof of Theorem \ref{main_th} (1)] 
Assume (A1) and (AS) and let $\mu > 0$. Due to Theorem \ref{main_th} (2),  P$(\mu, u_0,  v_0)$ has a unique weak solution $u$ on $[0,T]$. 
Since $u \in W(T)$, it is obvious that $f(u_x - 1) \in W^{1,2}(0,T;  H) \cap L^{\infty}(0,T; H^1(0,1))$ and $\sigma_b(u(\cdot, 0)) \in W^{1,2}(0,T)$. 
On account of  Lemma \ref{lem_aux}(2), (LP)$(\mu, u_0, v_0, g , f(u_x -1), \sigma_b(u(\cdot, 0)) )$ has a unique strong solution $\hat{u}  \in S(T)$.  On the other hand, $u$ is the weak solution of (LP)$(\mu, u_0, v_0, g , f(u_x -1), \sigma_b(u(\cdot, 0)) )$.  Hence, the uniqueness of the weak solutions implies that $u = \hat{u}$. This shows the conclusion of Theorem \ref{main_th} (1). 
\end{proof}

\section{Problem without viscosity} \label{last}
The aim of this section is to  prove Theorem \ref{main_th} (3). 
For this aim, we provide a lemma for a linear problem in case $\mu = 0$,  similarly to the previous section. 
To do so, the linear problem (LP)$(0, u_0,  v_0, g, F, q)$ denotes the system \eqref{EQA1} $\sim$ \eqref{ICA} with $\mu  = 0$.  

\begin{lem} \label{lem2_aux}
Let $\gamma > 0$ and $g \in {\mathbb R}$.  
If $u_0 \in H^2(0,1)$, $v_0 \in H$, $F \in L^2(0, T; H)$ and $q \in L^2(0, T)$, then the problem  (LP)$(0, u_0,  v_0, g, F, q)$ has a unique weak solution $u$ on $[0,T]$.  Namely, $u \in W_0(T)$, $u(0) = u_0$ and 
\begin{align}
 & - \int_{Q(T)} u_t \eta_t dxdt + \gamma \int_{Q(T)} u_{xx} \eta_{xx} dxdt - \int_{Q(T)} F \eta_x dxdt  \nonumber \\
 = & \int_0^1 v_0 \eta(0) dx + \int_{Q(T)} g \eta \ dxdt  - \int_0^T  q \eta(\cdot,0) dt \quad \mbox{ for } \eta \in {\mathcal T}_0(T).  
  \nonumber 
\end{align} 

Moreover, it holds that 
\begin{align}
 \frac{1}{2} \frac{d}{dt} |u_t|_H^2 + \frac{\gamma}{2} \frac{d}{dt} |u_{xx}|_H^2  
 \leq  \int_0^1 g u_t dx   - \int_0^1 F u_{tx} dx + q u_t(\cdot, 0) \quad \mbox{ a.e. on } [0,T].  \label{lem4-1-1} 
 \end{align}
 
\end{lem}

\begin{proof}
Lemma  \ref{lem_aux} implies that (LP)$(\mu, u_0,  v_0, g, F, q)$ has a weak solution $u_{\mu}$ on  $[0,T]$ for $\mu > 0$. 
Obviously, thanks to \eqref{lem3-1-1}, $\{u_{\mu}\}$ is bounded in $L^{\infty}(0, T; H^2(0,1))$ and $W^{1,2}(0, T; H)$.  
Hence,  taking a subsequence of $\{u_{\mu}\}$ we can prove the existence of a weak solution of (LP)$(0, u_0,  v_0, g, F, q)$ on. $[0,T]$.
The uniqueness is proved by the  dual equation method in the similar way to that in \cite{Aiki-Kosugi-2021}. 
Moreover, \eqref{lem4-1-1} can be shown from \eqref{lem3-1-1}.  
\end{proof}

\begin{proof}[Proof of Theorem \ref{main_th} (3)]
By Theorem \ref{main_th} (2) 
and \eqref{3-9}, P$(\mu, u_{0},  v_{0})$ has a weak solution $u_{\mu} \in W(T)$ for any $\mu  >0$ and it holds that 
\begin{align*} 
&  \frac{1}{2} |u_{\mu t}(t)|_H^2 + \frac{\gamma}{2} |u_{\mu xx}(t)|_H^2 + \mu \int_0^t |u_{\mu \tau x}|_H^2 d\tau + c_f \int_0^1 \frac{1}{|u_{\mu x}(t)|^2} dx 
- \hat{\sigma}_b(u_{\mu}(t,0)) \\
\leq  & C(u_0, v_0) + C_f + g \int_0^t  |u_{\mu \tau}|_H  d\tau \quad  \mbox{ for  } 0 \leq t \leq T.   
\end{align*}
On account of the Gronwall argument and Lemma \ref{esti_below} 
we see that the set $\{u_{\mu}\}$ is bounded in $L^{\infty}(0, T; H^2(0,1))$ and  $W^{1,\infty}(0,T; H)$, 
$u_{\mu x} \geq \delta_0 \mbox{ a.e. on } Q(T)$ and 
$u_{\mu}(t, 0) \geq  -\beta + \delta_0$ for a.e. $t \in [0,T]$, where $\delta_0$ is a positive constant independent of $\mu$. Moreover, 
$\{ \sqrt{\mu} u_{\mu \tau x}\}$ is bounded in $L^2(0, T; H)$. From these estimates, the Aubin compact theorem \cite{Lions} guarantees choice of a subsequence $\{\mu_j\}$ and $u \in L^{\infty}(0, T; H^2(0,1)) \cap W^{1,\infty}(0,T; H)$ such that $u_j (:= u_{\mu_j}) \to u$ weakly* in $L^{\infty}(0, T; H^2(0,1))$, $W^{1,\infty}(0,T; H)$ and strongly in $L^2(0, T; H^1(0,1))$ as $j \to \infty$. 
Also, by these convergences we have  
$u \geq \delta_0 \mbox{ a.e. on } Q(T)$ and $u(t, 0) \geq  -\beta + \delta_0$ for a.e. $t \in [0,T]$. Here, we note that 
\begin{align}
 & - \int_{Q(T)} u_{jt} \eta_t dxdt + \gamma \int_{Q(T)} u_{jxx} \eta_{xx} dxdt + \int_{Q(T)} f(u_{jx} - 1) \eta_x dxdt + \mu_j \int_{Q(T)} u_{jtx} \eta_x dxdt  
 \nonumber \\
 = & \int_0^1 v_0 \eta(0) dx  + \int_{Q(T)} g \eta \ dxdt - \int_0^T \sigma_b(u_j(t, 0)) \eta(t,0) dt 
\quad  \mbox{ for  } \eta \in {\mathcal T}_0(T) \mbox{ and } j.    \label{4-10}
\end{align} 
Hence, by the convergences, the estimate for $\sqrt{\mu} u_{\mu \tau x}$ and \eqref{4-10},  we can easily proved the existence of a weak solution. 

The proof of the uniqueness  is quite similar to that of Theorem \ref{main_th} (1). 
\end{proof} 

\begin{rmk}
\rm The proof of Theorem \ref{main_th} (3) given as above is quite  similar to that of \cite{Kosugi2}. 
\end{rmk}

\bibliographystyle{plain}
\bibliography{data}

\end{document}